\documentclass[12pt]{article}

\usepackage{graphicx}

\usepackage{amsmath}
\usepackage{amssymb}
\usepackage{amsfonts}
\usepackage{multirow}
\usepackage[backref=page]{hyperref}
\usepackage{xcolor}

\usepackage[margin=1in]{geometry}
\usepackage[width=12cm,format=plain,labelfont=sf,bf,small,textfont=sf,up,small,justification=justified,labelsep=period]{caption}
\usepackage{framed}

\hypersetup{
    colorlinks=true,       
    linkcolor=blue,        
    citecolor=blue,        
    filecolor=blue,        
    urlcolor=blue          
}
\usepackage{amsthm}
\theoremstyle{plain}
\newtheorem{proposition}{Proposition}
\newtheorem*{propstar}{Proposition}
\newtheorem{lemma}{Lemma}

\newtheorem{theorem}{Theorem}

\newtheorem{example}{Example}

\newcommand{\tr}{\textsf{T}}

\renewcommand{\bar}[1]{\overline{#1}}

\renewcommand{\tilde}[1]{\widetilde{#1}}

\DeclareMathOperator{\rank}{rank}

\DeclareMathOperator{\support}{support}

\DeclareMathOperator{\trace}{trace}

\DeclareMathOperator{\linspan}{span}
\DeclareMathOperator{\range}{range}

\newcommand{\cC}{\mathcal C}
\newcommand{\cF}{\mathcal F}
\newcommand{\cN}{\mathcal N}
\newcommand{\cS}{\mathcal S}
\newcommand{\cX}{\mathcal X}
\newcommand{\cY}{\mathcal Y}
\newcommand{\cV}{\mathcal V}

\newcommand{\NN}{\mathbb{N}}

\newcommand{\RR}{\mathbb{R}}

\renewcommand{\hat}[1]{\widehat{#1}}

\makeatletter
\def\iddots{\mathinner{\mkern1mu\raise\p@
\vbox{\kern7\p@\hbox{.}}\mkern2mu
\raise4\p@\hbox{.}\mkern2mu\raise7\p@\hbox{.}\mkern1mu}}
\makeatother

\title{Lower bounds on nonnegative rank\\ via nonnegative nuclear norms}

\author{Hamza Fawzi \and Pablo A. Parrilo\thanks{The authors are with
    the Laboratory for Information and Decision Systems, Department of
    Electrical Engineering and Computer Science, Massachusetts
    Institute of Technology, Cambridge, MA 02139. Email:
    \texttt{\{hfawzi,parrilo\}@mit.edu}. This research was funded in part by AFOSR FA9550-11-1-0305.}}
\renewcommand\footnotemark{}

\date{January 28th, 2015}

\begin{document}

\maketitle

\begin{abstract}
The nonnegative rank of an entrywise nonnegative matrix $A \in \RR^{m
  \times n}_+$ is the smallest integer $r$ such that $A$ can be
written as $A=UV$ where $U \in \RR^{m \times r}_+$ and $V \in \RR^{r
  \times n}_+$ are both nonnegative. The nonnegative rank arises in
different areas such as combinatorial optimization and communication
complexity. Computing this quantity is NP-hard in general and it is
thus important to find efficient bounding techniques especially in the
context of the aforementioned applications.

In this paper we propose a new lower bound on the nonnegative rank
which, unlike most existing lower bounds, does not solely rely on the
matrix sparsity pattern and applies to nonnegative matrices with
arbitrary support. The idea involves computing a certain nuclear norm with
nonnegativity constraints which allows to lower bound the nonnegative
rank, in the same way the standard nuclear norm gives lower bounds on
the standard rank. Our lower bound is expressed as the solution of a
copositive programming problem and can be relaxed to obtain
polynomial-time computable lower bounds using semidefinite
programming. We compare our lower bound with existing ones, and
we show examples of matrices where our lower bound performs better
than currently known ones.
\end{abstract}

\section{Introduction}

Given a nonnegative\footnote{Throughout the paper, a nonnegative matrix is a matrix whose entries are all nonnegative.} matrix $A \in \RR^{m \times n}_+$, the \emph{nonnegative rank} of $A$ is the smallest integer $r$ such that $A$ can be factorized as $A=UV$ where $U \in \RR^{m \times r}_+$ and $V \in \RR^{r \times n}_+$ are both nonnegative. The nonnegative rank of $A$ is denoted by $\rank_+(A)$ and it always satisfies:
\[ \rank(A) \leq \rank_+(A) \leq \min(n,m). \]

The nonnegative rank appears in different areas such as in combinatorial optimization \cite{yannakakis1991expressing} and communication complexity \cite{kushilevitz2006communication,lovasz1990communication,lee2009lower}. Indeed, in combinatorial optimization a well-known result by Yannakakis \cite{yannakakis1991expressing} shows that the nonnegative rank of a suitable matrix characterizes the smallest number of linear inequalities needed to represent a given polytope. This quantity is very important in practice since the complexity of interior-point algorithms for linear programming explicitly depends on the number of linear inequalities.
Another application of the nonnegative rank is in communication complexity where one is interested in the minimum number of bits that need to be exchanged between two parties in order to compute a binary function $f:\cX \times \cY \rightarrow \{0,1\}$, assuming that initially each party holds only one of the two arguments of the function. This quantity is known as the communication complexity of $f$ and is tightly related to the nonnegative rank of the $|\cX|\times |\cY|$ matrix $M_f$ associated to $f$ defined by $M_f(x,y) = f(x,y)$ \cite{lee2009lower,lovasz1990communication}.
Finally it was also recently observed \cite{zhang2012quantum,jain2013efficient} that the logarithm of the nonnegative rank of a matrix $A$ coincides with the minimum number of bits that need to be exchanged between two parties in order to sample from the bivariate probability distribution $p(x,y) = A_{x,y}$ represented by the matrix $A$ (assuming $A$ is normalized so that $\sum_{x,y} A_{x,y} = 1$).

\paragraph{Lower bounds on nonnegative rank} Unfortunately, the nonnegative rank is hard to compute in general unlike the standard rank: For example it was shown in \cite{vavasis2009complexity} that the problem of deciding whether $\rank_+(A) = \rank(A)$ is NP-hard in general (see also \cite{arora2012computing} for further hardness results). Researchers have therefore developed techniques to find lower and upper bounds for
this quantity. Existing lower bounds are not entirely satisfactory
though, since in general they depend only on the sparsity pattern of
the matrix $A$ and not on the actual values of the entries, and as a consequence
these bounds cannot be used when all the entries of $A$ are strictly positive. In fact most of the currently known lower bounds
are actually lower bounds on the \emph{rectangle covering number}
(also known as the \emph{Boolean rank}) which is a purely
combinatorial notion of rank that is itself a lower bound on the
nonnegative rank. The \emph{rectangle covering number} of a
nonnegative matrix $A \in \RR^{m \times n}_+$ is the smallest number
of rectangles needed to cover the nonzero entries of $A$. More
precisely it is the smallest integer $r$ such that there exist
$r$ rectangles $R_i = I_i \times J_i \subseteq \support(A)$ for
$i=1,\dots,r$ such that
\[ \support(A) = \bigcup_{i=1}^r R_i, \]
where $\support(A) = \{ (i,j) \in \{1,\dots,m\}\times \{1,\dots,n\} \; : \; A_{i,j} \neq 0
\}$ is the set of indices of the nonzero entries of $A$. It is easy to
see that the rectangle covering number is always smaller than or equal
to the nonnegative rank of $A$ (each nonnegative rank 1 term in a
decomposition of $A$ corresponds to a rectangle). Some of the
well-known lower bounds on the nonnegative rank such as the fooling
set method or the rectangle size method (see \cite[Section 1.3]{kushilevitz2006communication}) are in fact lower bounds on
the rectangle covering number and they only depend on the sparsity
pattern of $A$ and not on the specific values of its entries. Recently, a new lower bound was proposed in \cite{gillis2012geometric} that does not rely solely on the sparsity pattern; however the lower bound depends on a new quantity called the \emph{restricted nonnegative rank} which can be computed efficiently only when $\rank(A)\leq3$ but otherwise is NP-hard to compute in general. Also a non-combinatorial lower bound on the nonnegative rank known as the \emph{hyperplane separation} bound was proposed recently and used in the remarkable result of Rothvoss \cite{rothvoss2014matching} on the matching polytope. The lower bound we propose in this paper has a similar flavor as the hyperplane separation bound except that it uses the Frobenius norm instead of the entrywise infinity norm. Also one focus of the present paper is on computational approaches to compute the lower bound using sum-of-squares techniques and semidefinite programming. Finally, after the initial version of this paper was submitted, we extended some of the ideas presented here and we proposed in \cite{fawzi2014self} new lower bounds that are invariant under scaling and that are related to hyperplane separation bounds and combinatorial bounds.

\paragraph{Contribution} In this paper we present an efficiently computable lower bound on the nonnegative rank that does not rely exclusively on the sparsity pattern and that is applicable to matrices that are strictly positive. Before we present our result, recall that a \mbox{symmetric} matrix $M \in \RR^{n \times n}$ is said to be \emph{copositive} if $x^\tr M x \geq 0$ for all $x\in \RR^{n}_+$. Our result can be summarized in the following:

\medskip

\begin{theorem}
\label{thm:main}
Let $A \in \RR^{m \times n}_+$ be a nonnegative matrix. Let $\nu_+(A)$ be the optimal value of the following convex optimization program:
\begin{equation}
 \label{eq:nu+}
 \nu_+(A) = \max_{W \in \RR^{m \times n}} \; \left\{ \; \langle A, W \rangle \; : \; \begin{bmatrix} I & -W\\ -W^\tr & I \end{bmatrix} \text{ copositive} \; \right\}.
\end{equation}
Then we have
\begin{equation}
\label{eq:lb}
 \rank_+(A) \geq \left(\frac{\nu_+(A)}{\|A\|_F}\right)^2,
\end{equation}
where $\|A\|_F := \sqrt{\sum_{i,j} A_{i,j}^2}$ is the Frobenius norm of $A$.
\end{theorem}
Note that $\nu_+(A)$ is defined as the solution of a conic program over the cone of copositive matrices. Copositive programming is known to be NP-hard in general (see e.g., \cite{dur2010copositive}), but fortunately one can obtain good approximations using semidefinite programming.
For example one can obtain a lower bound to $\nu_+(A)$ by solving the following semidefinite program:
\begin{equation}
\label{eq:nu+0}
 \nu_+^{[0]}(A) = \max_{W \in \RR^{m \times n}} \; \left\{ \; \langle A, W \rangle \; : \; \begin{bmatrix} I & -W\\ -W^\tr & I \end{bmatrix} \in \cN^{n+m} + \cS^{n+m}_+ \; \right\}
\end{equation}
where $\cS^{n+m}_+$ denotes the cone of symmetric positive semidefinite matrices of size $n+m$, and $\cN^{n+m}$ denotes the cone of symmetric nonnegative matrices of size $n+m$. Since in general the sum of a nonnegative matrix and a positive semidefinite matrix is copositive, we immediately see that $\nu_+(A) \geq \nu_+^{[0]}(A)$ for any $A$, and thus this yields a polynomial-time computable lower bound to $\rank_+(A)$:
\[ \rank_+(A) \geq \left(\frac{\nu_+^{[0]}(A)}{\|A\|_F}\right)^2. \]
One can in fact obtain tighter estimates of $\nu_+(A)$ using
semidefinite programming by considering hierarchies of approximations of
the copositive cone, like e.g., the hierarchy developed in
\cite{parrilo2000structured}. This is discussed in more detail later
in the paper (Section~\ref{sec:copositive_hierarchy}).

Note from the definition of $\nu_+(A)$ (Equation \eqref{eq:nu+}) that $\nu_+(A)$ is convex in $A$ and is thus continuous on the interior of its domain. In Section \ref{sec:lbapproxnnrank} we show how the quantity $\nu_+(A)$ can in fact be used to obtain a lower bound on the nonnegative rank of any matrix $\tilde{A}$ that is $\epsilon$-close (in Frobenius norm) to $A$.

It is interesting to express the dual of the copositive program \eqref{eq:nu+}. The dual of the cone of copositive matrices is the cone of completely positive matrices. A symmetric matrix $M$ is said to be \emph{completely positive} if it admits a factorization $M=BB^\tr$ where $B$ is elementwise nonnegative. Note that a completely positive matrix is both nonnegative and positive semidefinite; however not every such matrix is necessarily completely positive (see e.g., \cite{berman2003completely} for an example and for more information on completely positive matrices). The dual of the copositive program \eqref{eq:nu+} defining $\nu_+(A)$ is the following completely positive program (both programs give the same optimal value by strong duality):
\[ \nu_+(A) = \min_{\substack{X \in \RR^{m \times m}\\ Y \in \RR^{n \times n}}} \; \left\{ \; \frac{1}{2}(\trace(X)+\trace(Y)) \; : \; \begin{bmatrix} X & A\\ A^\tr & Y \end{bmatrix} \text{ completely positive} \; \right\}. \]
We will revisit this completely positive program later in the paper in
Section~\ref{sec:nucnrm} when we discuss the relation between the
quantity $\nu_+(A)$ and the nuclear norm.

\paragraph{Outline} The paper is organized as follows. In Section~\ref{sec:lb} we give the proof of
the lower bound of Theorem~\ref{thm:main} and we also outline a
connection between the quantity $\nu_+(A)$ and the nuclear norm of a
matrix. We then discuss computational issues and we see how to obtain
semidefinite programming approximations of the quantity $\nu_+(A)$. In
Section~\ref{sec:examples} we look at specific examples of matrices and we show that
our lower bound can be greater than the plain rank lower bound and the rectangle covering number. In general however our lower bound is uncomparable to the
existing lower bounds, i.e., it can be either greater or smaller.
Among the examples, we show that our lower bound
is exact for the slack matrix of the hypercube, thus giving another
proof that the extension complexity of the hypercube in $n$ dimensions
is equal to $2n$ (a combinatorial proof of this fact is given in
\cite{fiorini2013combinatorial}).

\paragraph{Notations} Throughout the paper $\cS^n$ denotes the vector space of real symmetric matrices of size $n$ and $\cS^n_+$ is the cone of positive semidefinite
matrices of size $n$.
The cone of symmetric elementwise nonnegative matrices of size $n$ is denoted by $\cN^{n}$.
 If $X \in \cS^n$ we write $X \succeq 0$ to say that $X$ is positive semidefinite. Given
two matrices $X$ and $Y$, their inner product is defined as $\langle
X,Y \rangle = \trace(X^\tr Y) = \sum_{i,j} X_{i,j} Y_{i,j}$.

\section{The lower bound}
\label{sec:lb}

\subsection{Proof of the lower bound}
\label{sec:proof}

In this section we prove the lower bound of Theorem \ref{thm:main} in
a slightly more general form:

\begin{theorem}
\label{thm:main_gen}
Let $A \in \RR^{m \times n}_+$ be a nonnegative matrix. Let $P \in \RR^{m\times m}$ and $Q \in \RR^{n\times n}$ be nonnegative symmetric matrices with strictly positive diagonal entries. Let $\nu_+(A;P,Q)$ be the optimal value of the following copositive program:
\begin{equation}
 \label{eq:nu+PQ}
 \nu_+(A;P,Q) = \max_{W \in \RR^{m \times n}} \; \left\{ \; \langle A, W \rangle \; : \; \begin{bmatrix} P & -W\\ -W^\tr & Q \end{bmatrix} \text{ copositive} \; \right\}
\end{equation}
Then we have
\begin{equation}
\label{eq:lbPQ}
 \rank_+(A) \geq \left(\frac{\nu_+(A;P,Q)}{\sqrt{\trace(A^\tr P A Q)}}\right)^2.
\end{equation}
\end{theorem}

Note that Theorem \ref{thm:main} in the Introduction corresponds to
the special case where $P$ and $Q$ are the identity matrices. The more
general lower bound of Theorem \ref{thm:main_gen} has the advantage of
being invariant under diagonal scaling: namely if $\widetilde{A}=D_1 A
D_2$ is obtained from $A$ by positive diagonal scaling (where $D_1$
and $D_2$ are positive diagonal matrices), then the lower bound
produced by the previous theorem for both quantities $\rank_+(A)$ and
$\rank_+(\widetilde{A})$ (which are equal) will coincide, for adequate
choices of $P$ and $Q$. In fact it is easy to verify that if we choose
$P = D_1^{-2}$ and $Q=D_2^{-2}$ then
\[ \left(\frac{\nu_+(\widetilde{A};P,Q)}{\sqrt{\trace(\widetilde{A}^\tr P \widetilde{A} Q)}}\right)^2 = \left(\frac{\nu_+(A)}{\|A\|_F}\right)^2. \]
The dual of the copositive program defining $\nu_+(A;P,Q)$ is the following completely positive program where the matrices $P$ and $Q$ enter as weighting matrices in the objective function:
\[ \nu_+(A;P,Q) = \min_{\substack{X \in \RR^{m \times m}\\ Y \in \RR^{n \times n}}} \; \left\{ \; \frac{1}{2}(\trace(PX)+\trace(QY)) \; : \; \begin{bmatrix} X & A\\ A^\tr & Y \end{bmatrix} \text{ completely positive} \; \right\}. \]

\begin{proof}[Proof of Theorem \ref{thm:main_gen}]
We introduce the shorthand notations $\|x\|_P = \sqrt{x^\tr P x}$ when $x \in \RR^{m}_+$ and $\|y\|_Q = \sqrt{y^\tr Q y}$ when $y \in \RR^n_+$ which are well defined by the assumptions on $P$ and $Q$ (note however that $\|\cdot\|_P$ and $\|\cdot\|_Q$ are not necessarily norms in the usual sense).
Let $A$ be an $m \times n$ nonnegative matrix with nonnegative rank $r \geq 1$ and consider a factorization $A = UV = \sum_{i=1}^r u_i v_i^\tr$ where $u_i \in \RR^m_+$ are the columns of $U$ and $v_i^T \in \RR^n_+$ the rows of $V$. Clearly the vectors $u_i$ and $v_i$ are nonzero for all $i\in\{1,\dots,r\}$. By rescaling the $u_i$'s and $v_i$'s we can assume that $\|u_i\|_P = \|v_i\|_Q$ for all $i \in \{1,\dots,r\}$ (simply replace $u_i$ by $\tilde{u_i} = \gamma_i u_i$ and $v_i$ by $\tilde{v_i} = \gamma_i^{-1} v_i$ where $\gamma_i = \sqrt{\|v_i\|_Q / \|u_i\|_P}$). Observe that by the Cauchy-Schwarz inequality, we have:
\[ \frac{\sum_{i=1}^r \|u_i\|_P \|v_i\|_Q}{\sqrt{\sum_{i=1}^r \|u_i\|_P^2 \|v_i\|_Q^2}} \leq \sqrt{r} = \sqrt{\rank_+(A)} \]
We will now show separately that the numerator of the left-hand side above is lower bounded by the quantity $\nu_+(A;P,Q)$, and that the denominator is upper bounded by $\sqrt{\trace(A^\tr P A Q)}$:
\begin{itemize}
\item Numerator: If $W$ is such that $\left[\begin{smallmatrix} P & -W\\ -W^\tr & Q \end{smallmatrix}\right]$ is copositive then, since $u_i, v_i \geq 0$, we have:
\[ \begin{bmatrix} u_i\\ v_i \end{bmatrix}^\tr \begin{bmatrix} P & -W\\ -W^\tr & Q \end{bmatrix}\begin{bmatrix} u_i\\ v_i \end{bmatrix} \geq 0 \]
and hence
\[ u_i^\tr W v_i \leq \frac{1}{2}(\|u_i\|_P^2 + \|v_i\|_Q^2) = \|u_i\|_P \|v_i\|_Q \]
where we used the fact that $\|u_i\|_P = \|v_i\|_Q$.
Thus we get
\[ \langle A, W \rangle = \left\langle \sum_{i=1}^r u_i v_i^\tr, W \right\rangle = \sum_{i=1}^r u_i^\tr W v_i \leq \sum_{i=1}^r \|u_i\|_P \|v_i\|_Q. \]
Note that this is true for any $W$ such that
$\left[\begin{smallmatrix} P & -W\\ -W^\tr &
    Q \end{smallmatrix}\right]$ is copositive and thus we obtain
\[ \nu_+(A;P,Q) \leq \sum_{i=1}^r \|u_i\|_P \|v_i\|_Q. \]
\item Denominator: We now turn to finding an upper bound on $\sum_{i=1}^r  \|u_i\|_P^2 \|v_i\|_Q^2$. Observe that we have
\[
\begin{aligned} \trace(A^\tr P A Q) = \langle PA,AQ \rangle &= \sum_{1\leq i,j \leq r} \langle P u_i v_i^\tr, u_j v_j^\tr Q \rangle \\
               &= \sum_{i=1}^r \|u_i\|_P^2 \|v_i\|_Q^2 + \sum_{i \neq j} (u_i^\tr P u_j) (v_j^\tr Q v_i) \\
               &\geq \sum_{i=1}^r \|u_i\|_P^2 \|v_i\|_Q^2
\end{aligned}
\]
where in the last inequality we used the fact that $u_i^\tr P u_j \geq 0$ and $v_j^\tr Q v_j \geq 0$ which is true since $P$ and $Q$ are nonnegative.
\end{itemize}

Now if we combine the two points above we finally get the desired inequality
\[ \rank_+(A) = r \geq \left(\frac{\sum_{i=1}^r \|u_i\|_P \|v_i\|_Q}{\sqrt{\sum_{i=1}^r \|u_i\|_P^2 \|v_i\|_Q^2}}\right)^2 \geq \left(\frac{\nu_+(A;P,Q)}{\sqrt{\trace(A^\tr P A Q)}}\right)^2. \]
\end{proof}

\subsection{Lower bound on the approximate nonnegative rank}
\label{sec:lbapproxnnrank}

It is clear from the definition \eqref{eq:nu+} that the function \mbox{$A \in \RR^{m\times n}_+ \mapsto \nu_+(A)$} is convex and is thus continuous on the interior of its domain, unlike the nonnegative rank. A consequence of this is that the lower bound $(\nu_+(A)/\|A\|_F)^2$ will be small in general if $A$ is close to a matrix with small nonnegative rank. This continuity property of $\nu_+(A)$ can be used to obtain a lower bound on the nonnegative rank of any matrix $\tilde{A}$ that is close enough to $A$. Define the \emph{approximate nonnegative rank} of $A$, denoted $\rank_+^{\epsilon}(A)$, to be the smallest nonnegative rank among all nonnegative matrices that are $\epsilon$-close to $A$ in Frobenius norm:
\begin{equation}
\rank_+^{\epsilon}(A) = \min \left\{ \rank_+(\tilde{A}) \;\; : \;\; \tilde{A} \in \RR^{m \times n}_+ \text{ and } \|A - \tilde{A}\|_F \leq \epsilon \right\}.
\end{equation}
Approximate nonnegative factorizations and the approximate nonnegative rank have applications in lifts of polytopes \cite{gouveia2013approximate} as well as in information theory \cite{braun2013information}.

The following theorem shows that one can obtain a lower bound on $\rank_+^{\epsilon}(A)$ using the quantity $\nu_+(A)$:

\begin{theorem}
\label{thm:lbapproxnnrank}
Let $A \in \RR^{m \times n}_+$ be a nonnegative matrix and let $W$ be an optimal solution in the definition of $\nu_+(A)$ (cf. Equation~\eqref{eq:nu+}). Let $\epsilon$ be a positive constant with $\epsilon \leq \nu_+(A) / \|W\|_F$. Then we have:
\[ \rank_+^\epsilon(A) \geq \left(\frac{\nu_+(A) - \epsilon \|W\|_F}{\|A\|_F+\epsilon}\right)^2. \]
\end{theorem}
\begin{proof}
Let $\tilde{A}$ be any nonnegative matrix such that $\|A-\tilde{A}\|_F \leq \epsilon$. Since $W$ is a feasible point for the copositive program that defines $\nu_+(\tilde{A})$ we clearly have 
\[ \nu_+(\tilde{A}) \geq \langle \tilde{A}, W \rangle = \langle A, W \rangle + \langle \tilde{A}-A, W \rangle \geq \nu_+(A) - \epsilon \|W\|_F \]
Hence since $\nu_+(A) - \epsilon \|W\|_F \geq 0$ and $\|\tilde{A}\|_F \leq \|A\|_F + \epsilon$ we get:
\[ \rank_+(\tilde{A}) \geq \left(\frac{\nu_+(\tilde{A})}{\|\tilde{A}\|_F}\right)^2 \geq \left(\frac{\nu_+(A) - \epsilon \|W\|_F}{\|A\|_F+\epsilon}\right)^2. \]
Since this is valid for any $\tilde{A}$ that is $\epsilon$-close to $A$, we have
\[ \rank_+^\epsilon(A) \geq \left(\frac{\nu_+(A) - \epsilon \|W\|_F}{\|A\|_F+\epsilon}\right)^2. \]
\end{proof}
\subsection{Connection with nuclear norm}
\label{sec:nucnrm}

In this section we discuss the connection between the quantity $\nu_+(A)$ and nuclear norms of linear operators. If $A$ is an arbitrary (not necessarily nonnegative) matrix the \emph{nuclear norm} of $A$ is defined by  \cite{jameson1987summing}:
\begin{equation}
\label{eq:nu}
\nu(A) = \min \; \left\{ \; \sum_{i} \|u_i\|_2 \|v_i\|_2 \; : \; A = \sum_{i} u_i v_i^\tr \; \right\}.
\end{equation}
It can be shown that the quantity $\nu(A)$ above is equal to the sum of the singular values $\sigma_1(A)+\dots+\sigma_r(A)$ of $A$, and in fact this is the most commonly encountered definition of the nuclear norm. This latter characterization gives the following well-known lower bound on $\rank(A)$ combining $\nu(A)$ and $\|A\|_F$:
\[ \rank(A) \geq \left(\frac{\sigma_1(A)+\dots+\sigma_r(A)}{\sqrt{\sigma_1(A)^2+\dots+\sigma_r(A)^2}}\right)^2 = \left(\frac{\nu(A)}{\|A\|_F}\right)^2. \]

The characterization of nuclear norm given in Equation~\eqref{eq:nu} can be very naturally adapted to nonnegative factorizations of $A$ by restricting the vectors $u_i$ and $v_i$ in the decomposition of $A$ to be nonnegative. The new quantity that we obtain with this restriction is in fact nothing but the quantity $\nu_+(A)$ introduced earlier, as we show in the Theorem below. Observe that this quantity $\nu_+(A)$ is always greater than or equal than the standard nuclear norm $\nu(A)$.

\begin{theorem}
\label{thm:nucnrm}
Let $A \in \RR^{m \times n}_+$ be a nonnegative matrix. The following three quantities are equal to $\nu_+(A)$ as defined in Equation~\eqref{eq:nu+}:
\begin{enumerate}
\item[(i)] \hspace{0.25cm} $\displaystyle\min \; \left\{ \; \sum_{i} \|u_i\|_2 \|v_i\|_2 \; : \; A = \sum_{i} u_i v_i^\tr, \; u_i, v_i \geq 0 \; \right\}$
\item[(ii)] $\displaystyle\min_{\substack{X \in \cS^{m}\\ Y \in \cS^n}} \; \left\{ \; \frac{1}{2}(\trace(X)+\trace(Y)) \; : \; \begin{bmatrix} X & A\\ A^\tr & Y \end{bmatrix} \text{ completely positive} \; \right\}$
\item[(iii)] $\displaystyle\max_{W \in \RR^{m \times n}} \; \left\{ \; \langle A, W \rangle \; : \; \begin{bmatrix} I & -W\\ -W^\tr & I \end{bmatrix} \text{ copositive} \; \right\}$
\end{enumerate}
\end{theorem}

\begin{proof}
Note that the conic programs in $(ii)$ and $(iii)$ are dual of each other; furthermore program $(iii)$ is strictly feasible (simply take $W=0$) thus it follows from strong duality that $(ii)$ and $(iii)$ have the same optimal value. We thus have to show only that $(i)$ and $(ii)$ are equal.

We start by proving that $(ii) \leq (i)$: If $A = \sum_{i=1}^k u_i v_i^\tr$ where $u_i,v_i \geq 0$ and $u_i,v_i \neq 0$ for all $i\in\{1,\dots,k\}$, then if we let 
\[ X = \sum_{i=1}^k \frac{\|v_i\|_2}{\|u_i\|_2} u_i u_i^\tr \quad \text{and} \quad Y = \sum_{i=1}^k \frac{\|u_i\|_2}{\|v_i\|_2} v_i v_i^\tr \]
then
\[ \begin{bmatrix} X & A\\ A^\tr & Y \end{bmatrix} = \sum_{i=1}^k \|u_i\|_2 \|v_i\|_2 \begin{bmatrix}  u_i / \|u_i\|_2 \\ v_i / \|v_i\|_2 \end{bmatrix} \begin{bmatrix}  u_i / \|u_i\|_2 \\ v_i / \|v_i\|_2 \end{bmatrix}^\tr \]
and so $\left[ \begin{smallmatrix} X & A\\ A^\tr & Y \end{smallmatrix} \right]$ is completely positive since $u_i$ and $v_i$ are elementwise nonnegative. Furthermore we have
\[ \trace(X)+\trace(Y) = \sum_{i=1}^k \|u_i\|_2 \|v_i\|_2 + \sum_{i=1}^k \|u_i\|_2 \|v_i\|_2 = 2 \sum_{i=1}^k \|u_i\|_2 \|v_i\|_2. \]
Hence this shows that $(ii) \leq (i)$.

We now show that $(i) \leq (ii)$. Assume that $X$ and $Y$ are such that $\left[ \begin{smallmatrix} X & A\\ A^\tr & Y \end{smallmatrix} \right]$ is completely positive and consider a decomposition 
\[ \begin{bmatrix} X & A\\ A^\tr & Y \end{bmatrix} = \sum_{i=1}^k \begin{bmatrix} x_i\\ y_i\end{bmatrix} \begin{bmatrix} x_i\\ y_i\end{bmatrix}^\tr \]
where $x_i$ and $y_i$ are nonnegative. Then from this decomposition we have $A = \sum_{i=1}^k x_i y_i^\tr$ and
\[ \sum_{i=1}^k \|x_i\|_2 \|y_i\|_2 \leq \frac{1}{2} \left(\sum_{i=1}^k \|x_i\|_2^2 + \sum_{i=1}^k \|y_i\|_2^2\right) = \frac{1}{2} (\trace(X)+\trace(Y)). \]
This shows that $(i) \leq (ii)$ and completes the proof.
\end{proof}

For an arbitrary (not necessarily nonnegative) matrix $A$ the nuclear norm $\nu(A)$ can be computed as the optimal value of the following primal-dual pair of semidefinite programs:
\begin{equation}
\label{eq:svdpair}
\begin{minipage}{0.5\linewidth}
	\begin{array}[t]{ll}
	\underset{\substack{X \in \cS^{m}\\ Y \in \cS^{n}}}{\mbox{minimize}}   & \frac{1}{2} (\trace(X)+\trace(Y)) \\[0.3cm]
	\mbox{subject to} & \begin{bmatrix} X & A\\ A^\tr & Y \end{bmatrix} \in \cS^{m+n}_+
	\end{array}
\end{minipage}\hspace{0.5cm}\vrule\hspace{0.5cm}
\begin{minipage}{0.5\linewidth}
	\begin{array}[t]{ll}
	\underset{W \in \RR^{m \times n}}{\mbox{maximize}}   & \langle A, W \rangle  \\[0.3cm]
	\mbox{subject to} & \begin{bmatrix} I & -W\\ -W^\tr & I \end{bmatrix}  \in \cS^{m+n}_+
	\end{array}
\end{minipage}
\end{equation}
A solution of these semidefinite programs can be obtained from a singular value decomposition of $A$: namely if $A=U\Sigma V^\tr$ is a singular value decomposition of $A$, then $X=U\Sigma U^\tr$, $Y=V\Sigma V^\tr$ and $W=UV^\tr$ are optimal points of the semidefinite programs above, and the optimal value is $\trace(\Sigma)=\nu(A)$.

Note that the primal-dual pairs that define $\nu(A)$ and $\nu_+(A)$ are very similar, and the only differences are in the cones used: for $\nu(A)$ it is the self-dual cone of positive semidefinite matrices, whereas for $\nu_+(A)$ it is the dual pair of completely positive / copositive cones. Note also that these optimization problems are related to the convex programs that arise in \cite{doan2013finding}. In the following proposition we give simple sufficient conditions on the singular value decomposition of $A$ that allow to check if the two quantities $\nu_+(A)$ and $\nu(A)$ are actually equal. For the proposition recall that $\nu_+^{[0]}(A)$ is the first approximation of $\nu_+(A)$ defined in \eqref{eq:nu+0}, where the copositive cone is replaced by the cone $\cN^{n+m}+\cS^{n+m}_+$. Using duality the quantity $\nu_+^{[0]}(A)$ is also equal to the solution of the following minimization problem:
\begin{equation}
\label{eq:dualnu+0}
 \min_{X,Y} \; \left\{ \; \frac{1}{2}(\trace(X)+\trace(Y)) \; : \; \begin{bmatrix} X & A\\ A^\tr & Y \end{bmatrix} \in \cS^{2n}_+ \cap \cN^{2n} \; \right\}
\end{equation}
where $\cS^{2n}_+ \cap \cN^{2n}$ is the cone of doubly nonnegative matrices (i.e., the cone of matrices that are nonnegative and positive semidefinite).

\begin{proposition}
Let $A \in \RR^{m \times n}_+$ be a nonnegative matrix with a singular value decomposition $A = U\Sigma V^\tr$ and let $\nu(A) = \trace(\Sigma)$ be the nuclear norm of $A$. \\
(i) If $U\Sigma U^\tr$ and $V\Sigma V^\tr$ are nonnegative then $\nu_+^{[0]}(A) = \nu(A)$.\\
(ii) Also, if the matrix
\begin{equation}
 \label{eq:blockSVDmatrix}
 \begin{bmatrix}
U\Sigma U^\tr & A\\
A^\tr & V\Sigma V^\tr
\end{bmatrix} \in \cS^{n+m}
\end{equation}
is completely positive, then $\nu_+(A) = \nu(A)$.
\end{proposition}
\begin{proof}
To prove point (i) of the proposition note first that we always have the inequality $\nu_+^{[0]}(A) \geq \nu(A)$. Now if $U\Sigma U^\tr$ and $V\Sigma V^\tr$ are nonnegative then $X=U\Sigma U^\tr$ and $Y=V\Sigma V^\tr$ are feasible for \eqref{eq:dualnu+0} and achieve the value $\trace(\Sigma) = \nu(A)$. This shows that in this case $\nu_+^{[0]}(A) \leq \nu(A)$ and thus $\nu_+^{[0]}(A) = \nu_+(A)$. The proof of item (ii) is similar.
\end{proof}
To finish this section we investigate properties of the solution of the completely positive/copositive pair that defines $\nu_+(A)$ (cf. (ii) and (iii) in Theorem \ref{thm:nucnrm}).
An interesting question is to know whether the solution of this primal-dual pair possesses some interesting properties like the singular value decomposition. The next proposition shows a property which relates the optimal $W$ of the copositive program and a nonnegative decomposition $A=\sum_i u_i v_i^\tr$ satisfying $\sum_i \|u_i\|_2 \|v_i\|_2 = \nu_+(A)$.

\begin{proposition}
\label{prop:optimalW}
Let $A \in \RR^{m \times n}_+$ be a nonnegative matrix. Assume $A=\sum_i \lambda_i u_i v_i^\tr$ is a nonnegative decomposition of $A$ that satisfies
\[ \sum_i \lambda_i = \nu_+(A) \]
where $\lambda_i \geq 0$ and $u_i \in \RR^m_+, v_i \in \RR^n_+$ have unit norm $\|u_i\|_2 = \|v_i\|_2 = 1$. Let $W$ be an optimal point in the copositive program \eqref{eq:nu+}. Then we have for all $i$, 
\[ u_i = \Pi_+(Wv_i) \quad \text{and} \quad v_i = \Pi_+(W^\tr u_i) \]
where $\Pi_+(x) = \max(x,0)$ is the projection on the nonnegative orthant.
\end{proposition}

\noindent The proposition above can be seen as the analogue in the nonnegative world of the fact
that, for arbitrary matrices $A$, the optimal $W$ in the semidefinite
program \eqref{eq:svdpair} satisfies $u_i = Wv_i$ and $v_i = W^\tr
u_i$ where $u_i$ and $v_i$ are respectively the left and right singular
vectors of $A$.

Before proving the proposition we prove the following simple lemma:

\begin{lemma}
\label{lem:copositive_operator_norm}
If $W \in \RR^{m \times n}$ is such that $\begin{bmatrix} I & -W\\ -W^\tr & I \end{bmatrix}$ is copositive, then for any $v \in \RR^n_+$ we have \[ \|\Pi_+(Wv)\|_2 \leq \|v\|_2 \]
where $\Pi_+(x) = \max(x,0)$ is the projection on the nonnegative orthant.
\end{lemma}
\begin{proof}
Let $v \in \RR^n_+$ and call $u = \Pi_+(Wv)$. By the copositivity assumption we have, since $u$ and $v$ are nonnegative:
\[ \begin{bmatrix} u \\ v\end{bmatrix}^\tr \begin{bmatrix} I & -W\\ -W^\tr & I \end{bmatrix} \begin{bmatrix} u \\ v \end{bmatrix} \geq 0 \]
which can rewritten as:
\[ (u-Wv)^\tr (u-Wv) + v^\tr (I-W^\tr W) v \geq 0. \]
Since $u=\Pi_+(Wv)$ we get:
\[ \|\Pi_+(Wv) - Wv\|_2^2 + \|v\|_2^2 - \|Wv\|_2^2 \geq 0. \]
Note that $\|\Pi_+(Wv) - Wv\|_2^2 - \|Wv\|_2^2 = -\|\Pi_{+}(Wv)\|_2^2$ and thus we get the desired inequality
\[ \|\Pi_{+}(Wv)\|_2^2 \leq \|v\|_2^2. \]
\end{proof}

\noindent Using this lemma we now prove Proposition \ref{prop:optimalW}:

\begin{proof}[Proof of Proposition \ref{prop:optimalW}]
Let $A=\sum_i \lambda_i u_i v_i^\tr$ be a nonnegative decomposition of $A$ that satisfies $\sum_i \lambda_i = \nu_+(A)$ and let $W$ be an optimal point of the copositive program \eqref{eq:nu+}. Since $\begin{bmatrix} I & -W\\ -W^\tr & I \end{bmatrix}$ is copositive and $\|u_i\|_2 = \|v_i\|_2 = 1$, we have for each $i$, $u_i^\tr W v_i \leq 1$. Now since $W$ is optimal we have 
\[ \sum_i \lambda_i u_i^\tr W v_i = \langle A, W \rangle = \nu_+(A) = \sum_i \lambda_i, \]
and hence for each $i$ we have necessarily $u_i^\tr W v_i = 1$. Furthermore, we have the sequence of inequalities
\begin{equation}
\label{eq:inequalities}
 1 = u_i^\tr W v_i \leq u_i^\tr \Pi_+(Wv_i) \leq \|u_i\|_2 \|\Pi_+(Wv_i)\|_2 \leq \|u_i\|_2 \|v_i\|_2 = 1
\end{equation}
where in the first inequality we used that $u \geq 0$, then we used Cauchy-Schwarz inequality and for the third inequality we used Lemma \ref{lem:copositive_operator_norm}. Since the left-hand side and the right-hand side of \eqref{eq:inequalities} are equal this shows that all the intermediate inequalities are in fact equalities. In particular by the equality case in Cauchy-Schwarz we have that $\Pi_+(Wv_i) = \rho_i u_i$ for some constant $\rho_i$. But since $u_i^\tr \Pi_+(Wv_i) = 1$ and $\|u_i\|_2 = 1$ we get that $\rho_i = 1$. This shows finally that $u_i = \Pi_+(Wv_i)$. To prove that $v_i = \Pi_+(W^\tr u_i)$ we use the same line of inequalities as in \eqref{eq:inequalities} starting from the fact that $1 = v_i^\tr W^\tr u_i$.
\end{proof}

\subsection{Approximations and semidefinite programming lower bounds}
\label{sec:copositive_hierarchy}

In this section we describe how one can use semidefinite programming
to obtain polynomial-time computable lower bounds on $\nu_+(A)$, and
thus on $\rank_+(A)$. We outline here the hierarchy of approximations of
the copositive cone proposed by Parrilo in
\cite{parrilo2000structured} using
sums-of-squares techniques.

Recall that a symmetric matrix $M \in \cS^n$ is copositive if $x^\tr M x \geq 0$ for all $x \geq 0$. An equivalent way of formulating this condition is to say that the following polynomial of degree~4
\[ \sum_{1\leq i,j \leq n} M_{i,j} x_i^2 x_j^2 \]
is globally nonnegative. The cone $\cC$ of copositive matrices can thus be described as:
\[ \cC = \left\{ M \in \cS^n \; : \; \text{the polynomial } \sum_{1\leq i,j \leq n} M_{i,j}  x_i^2 x_j^2 \text{ is nonnegative} \right\}. \]
The $k$'th order inner approximation of $\cC$ proposed by Parrilo in \cite{parrilo2000structured} is defined as:
\begin{equation}
\label{eq:cCk}
 \cC^{[k]} = \left\{ M \in \cS^n \; : \; \left(\sum_{i=1}^n x_i^2\right)^{k} \left(\sum_{1\leq i,j \leq n} M_{i,j} x_i^2 x_j^2\right) \text{ is a sum of squares} \right\}. 
\end{equation}
It is clear that for any $k$ we have $\cC^{[k]} \subseteq \cC$ and also $\cC^{[k]} \subseteq \cC^{[k+1]}$, thus the sequence $(\cC^{[k]})_{k \in \NN}$ forms a sequence of nested inner approximations to the cone of copositive matrices; furthermore it is known via P{\'o}lya's theorem that this sequence converges to the copositive cone, cf. \cite{parrilo2000structured}. A crucial property of this hierarchy is that each cone $\cC^{[k]}$ can be represented using linear matrix inequalities, which means that optimizing a linear function over any of these cones is equivalent to a semidefinite program. Of course, the size of the semidefinite program gets larger as $k$ gets larger, and the size of the semidefinite program at the $k$'th level is  $\binom{n+k+1}{k+2}$. Note however that the semidefinite programs arising from this hierarchy can usually be simplified so that they can be solved more efficiently, cf. \cite{parrilo2000structured} and \cite[Section 8.1]{gatermann2004symmetry}.

An interesting feature concerning this hierarchy is that the approximation of order $k=0$ corresponds to 
\[ \cC^{[0]} = \cN^{n} + \cS^{n}_+, \]
i.e., $\cC^{[0]}$ is the set of matrices that can be written as the sum of a nonnegative matrix and a positive semidefinite matrix (this particular approximation was already mentioned in the introduction).

Using this hierarchy, we can now compute lower bounds to $\rank_+(A)$ using semidefinite programming:
\begin{theorem}
\label{thm:mainrelax}
Let $A \in \RR^{m \times n}_+$ be a nonnegative matrix. For $k \in \NN$, let $\nu_+^{[k]}(A)$ be the optimal value of the following semidefinite programming problem:
\begin{equation}
 \label{eq:nu+k}
 \nu_+^{[k]}(A) = \max \; \left\{ \; \langle A, W \rangle \; : \; \begin{bmatrix} I & -W\\ -W^\tr & I \end{bmatrix} \in \cC^{[k]} \; \right\}
\end{equation}
where the cone $\cC^{[k]}$ is defined in Equation~\eqref{eq:cCk}.
Then we have:
\begin{equation}
\label{eq:lbk}
 \rank_+(A) \geq \left(\frac{\nu_+^{[k]}(A)}{\|A\|_F}\right)^2.
\end{equation}
\end{theorem}

As $k$ gets larger the quantity $(\nu_+^{[k]}(A) / \|A\|_F)^2$ will
converge (from below) to $(\nu_+(A) / \|A\|_F)^2$ (where $\nu_+(A)$ is
defined by the copositive program \eqref{eq:nu+}) and this quantity
can be strictly smaller than $\rank_+(A)$ as can be seen for instance
in Example \ref{ex:derangement} given later in the paper. In general,
the quantity $(\nu_+^{[k]}(A) / \|A\|_F)^2$ may not converge to
$\rank_+(A)$, however it will always be a valid lower bound to
$\rank_+(A)$. In summary, we can write for any $k \geq 0$:
\[ \nu(A) \leq \nu_+^{[0]}(A) \leq \nu_+^{[k]}(A) \leq \nu_+(A) \leq \sqrt{\rank_+(A)} \| A \|_F \]
where $\nu(A)$ is the standard nuclear norm of $A$. The lower bounds \eqref{eq:lbk} are thus always greater than the lower bound that uses the standard nuclear norm (in fact they can also be greater than the standard rank lower bound as we show in the examples later).

On the webpage \url{http://www.mit.edu/~hfawzi} we provide a MATLAB script to compute the quantity $\nu_+^{[k]}(A)$ and the associated lower
bound on $\rank_+(A)$ for any nonnegative matrix $A$ and approximation
level $k \in \NN$. The script uses the software YALMIP and its
Sum-Of-Squares module in order to compute the quantity
$\nu_+^{[k]}(A)$ \cite{yalmip,yalmipSOS}. On a standard computer, the script allows to compute the lower bound of level $k=0$ for matrices up to size $\approx 50$.

Note that another hierarchy of inner approximations of the copositive cone has
been proposed by de Klerk and Pasechnik in
\cite{klerk2002approximation}. This hierarchy is based on linear
programming (instead of semidefinite programming) and it leads in
general to smaller programs that can be solved more
efficiently. However the convergence of the linear programming hierarchy is in general
much slower than the semidefinite programming hierarchy and one typically needs to take
very large values of $k$ (the hierarchy level) to obtain interesting
bounds.

\paragraph{Program size reduction when $A$ is symmetric}
If the nonnegative matrix $A$ is square and symmetric, the
semidefinite programs that define $\nu_+^{[k]}(A)$ can be simplified
to obtain smaller optimization problems that are easier to solve. For simplicity, we describe below in detail the case of $k=0$; the extensions to the general case can also be done. When $A$ is symmetric, the dual program defining $\nu_+^{[0]}(A)$ is:
\[ \text{minimize} \;\; \frac{1}{2} (\trace(X)+\trace(Y)) \;\; \text{subject to} \;\; \begin{bmatrix} X & A\\ A & Y \end{bmatrix} \in \cS^{2n}_+ \cap \cN^{2n} \]
where $\cS^{2n}_+ \cap \cN^{2n}$ is the cone of doubly nonnegative matrices. It is not difficult to see that one can always find an optimal solution of the SDP above where $X=Y$. Indeed if the matrix $\left[\begin{smallmatrix} X & A\\ A & Y \end{smallmatrix}\right]$ is feasible, then so is the matrix $\left[\begin{smallmatrix} Y & A\\ A & X \end{smallmatrix}\right]$ which also achieves the same objective function and thus by averaging we get a feasible point with the same matrices on the diagonal and with the same objective function. Therefore for symmetric $A$, the quantity $\nu_+^{[0]}(A)$ is given by the optimal value of:
\[ \text{minimize} \;\; \trace(X) \;\; \text{subject to} \;\; \begin{bmatrix} X & A\\ A & X \end{bmatrix} \in \cS^{2n}_+ \cap \cN^{2n}. \]
Now using the known fact that (see e.g., \cite[Fact 8.11.8]{bernstein2009matrix}) 
\[ \begin{bmatrix} X & A\\ A & X \end{bmatrix} \succeq 0 \;\; \Leftrightarrow \;\; X-A \succeq 0 \text{ and } X+A \succeq 0 \]
we can rewrite the program as:
\[ \text{minimize} \;\; \trace(X) \;\; \text{subject to} \;\; \begin{cases} X-A \succeq 0 & \\ X+A \succeq 0 & \\  X \text{ nonnegative} \end{cases} \]
This new program has two smaller positive semidefinite constraints, each of size $n$, and thus can be solved more efficiently than the previous definition which involved one large positive semidefinite constraint of size $2n$. Indeed most current interior-point solvers exploit block-diagonal structure in semidefinite programs.

One can also perform the corresponding symmetry reduction for the \emph{primal} program that defines $\nu_+^{[0]}(A)$, and this gives:
\begin{equation}
 \text{maximize} \;\; \langle A, W \rangle \;\; \text{subject to} \;\; \begin{cases} R-W \succeq 0 & \\ R+W \succeq 0 & \\ I-R \text{ nonnegative} \end{cases}
 \label{eq:reducedSDPdual}
\end{equation}
Observe that the constraint in the program above implies that $\left[\begin{smallmatrix} I & -W\\ -W^\tr & I\end{smallmatrix}\right] \in \cS^{2n}_+ + \cN^{2n}$ since:
\[ \begin{bmatrix} I & -W\\ -W & I\end{bmatrix} = \begin{bmatrix} I-R & 0\\ 0 & I-R\end{bmatrix} + \begin{bmatrix} R & -W\\ -W & R\end{bmatrix} \]
where in the right-hand side the first matrix is nonnegative and the second one is positive semidefinite.

\section{Examples}
\label{sec:examples}

In this section we apply our lower bound to some explicit matrices and we compare it to existing lower bounding techniques.

\begin{example}
\label{ex:cohenrothblum}
We start with an example where our lower bound exceeds the plain rank lower bound. Consider the following $4\times 4$ nonnegative matrix from \cite{cohen1993nonnegative}:
\[ 
\setlength{\arraycolsep}{3pt}
A = \left[\begin{array}{rrrr}
1 & 1 & 0 & 0\\
1 & 0 & 1 & 0\\
0 & 1 & 0 & 1\\
0 & 0 & 1 & 1
\end{array}\right]
\]
The rank of this matrix is 3 and its nonnegative rank is 4 as was noted in \cite{cohen1993nonnegative}. We can compute the exact value of $\nu_+^{[0]}(A)$ for this matrix $A$ and we get $\nu_+^{[0]}(A) = 4\sqrt{2} \approx 5.65$. We thus see that the lower bound is sharp for this matrix:
\[ 4 = \rank_+(A) \geq \left(\frac{\nu_+^{[0]}(A)}{\|A\|_F}\right)^2 = \left(\frac{4\sqrt{2}}{\sqrt{8}}\right)^2 = 4. \]
The optimal matrix $W$ in the semidefinite program\footnote{Note that the matrix considered in this example is symmetric, and so one could use the reduced semidefinite program \eqref{eq:reducedSDPdual} to simplify the computation of $\nu_+^{[0]}(A)$. However for simplicity and for illustration purposes, we used in this example the original formulation \eqref{eq:nu+0}.} \eqref{eq:nu+0} for which $\nu_+^{[0]}(A) = \langle A, W \rangle$ is given by:
\begin{equation}
\label{eq:Wsquare}
W = \frac{1}{\sqrt{2}}
\left[
\begin{array}{rrrr}
1 & 1 & -1 & -1\\
1 & -1 & 1 & -1\\
-1 & 1 & -1 & 1\\
-1 & -1 & 1 & 1
\end{array}
\right]
=
\frac{1}{\sqrt{2}}(2A - J)
\end{equation}
Observe that the matrix $W$ is obtained from $A$ by replacing the ones with $\frac{1}{\sqrt{2}}$ and the zeros with $-\frac{1}{\sqrt{2}}$.
The matrix $W$ is feasible for the semidefinite program \eqref{eq:nu+0} and one can check that we have the following decomposition of $\begin{bmatrix} I & -W\\ -W^\tr & I\end{bmatrix}$ into a nonnegative part and a positive semidefinite part:

{
\setlength{\arraycolsep}{2pt}
\small
\begin{equation}
\label{eq:P+Nsquare}
\left[
\begin{array}{c|c}
\begin{array}{rrrr} \phantom{0} & \phantom{0}  & \phantom{0}  & \phantom{0} \\  & \multicolumn{2}{c}{\multirow{2}{*}{$I$}} &  \\  &  &  &  \\ \phantom{0} &  \phantom{0} &  \phantom{0} & \phantom{0} \end{array} & -W\\ \hline
-W^\tr & \begin{array}{rrrr} \phantom{0} & \phantom{0}  & \phantom{0}  & \phantom{0} \\  & \multicolumn{2}{c}{\multirow{2}{*}{$I$}} &  \\  &  &  &  \\ \phantom{0} &  \phantom{0} &  \phantom{0} & \phantom{0} \end{array}
\end{array}
\right]
=
\underbrace{
\left[
\begin{array}{c|c}
\begin{array}{rrrr}0 & 0 & 0 & 1\\0 & 0 & 1 & 0\\0 & 1 & 0 & 0\\1 & 0 & 0 & 0\end{array} & 0\\ \hline
 0 & \begin{array}{rrrr}0 & 0 & 0 & 1\\0 & 0 & 1 & 0\\0 & 1 & 0 & 0\\1 & 0 & 0 & 0\end{array}
\end{array}
\right]
}_{\text{\normalsize nonnegative}}
+
\underbrace{
\left[
\begin{array}{c|c}
\begin{array}{rrrr}1 & 0 & 0 & -1\\0 & 1 & -1 & 0\\0 & -1 & 1 & 0\\-1 & 0 & 0 & 1\end{array} & -W\\ \hline
-W^\tr & \begin{array}{rrrr}1 & 0 & 0 & -1\\0 & 1 & -1 & 0\\0 & -1 & 1 & 0\\-1 & 0 & 0 & 1\end{array}
\end{array}
\right]
}_{\text{\normalsize positive semidefinite}}
\end{equation}
}
\end{example}

\begin{example}[Slack matrix of the hypercube]
\label{ex:slackhypercube}
 The $4\times 4$ matrix of the previous example is in fact the \emph{slack matrix} of the square $[0,1]^2$ in the plane. Recall that the slack matrix \cite{yannakakis1991expressing} of a polytope $P \subset \RR^n$ with $f$ facet inequalities $b_i - a_i^\tr x \geq 0$, $i=1,\dots,f$ and $v$ vertices $x_1,\dots,x_v \in \RR^n$ is the $f\times v$ matrix $S(P)$ given by:
\[ S(P)_{i,j} = b_i - a_i^\tr x_j \quad \forall i=1,\dots,f, \; j=1,\dots,v. \]
The matrix $S(P)$ is clearly nonnegative since the vertices $x_j$ belong to $P$ and satisfy the facet inequalities. Yannakakis showed in \cite{yannakakis1991expressing} that the nonnegative rank of $S(P)$ coincides with the smallest number of linear inequalities needed to represent the polytope\footnote{Note that the trivial representation of $P$ uses $f$ linear inequalities, where $f$ is the number of facets of $P$. However by introducing new variables (i.e., allowing projections) one can sometimes reduce dramatically the number of inequalities needed to represent $P$. For example the cross-polytope $P = \{x \in \RR^n : w^\tr x \leq 1\; \forall w \in \{-1,1\}^n\}$ has $2^n$ facets but can be represented using only $2n$ linear inequalities after introducing $n$ additional variables: $P = \{x \in \RR^n \; : \; \exists y \in \RR^n \; y \geq x, \; y \geq -x, \; 1^\tr y = 1\}$.} $P$. The minimal number of linear inequalities needed to represent $P$ is also known as the extension complexity of $P$; the theorem of Yannakakis therefore states that $\rank_+(S(P))$ is equal to the extension complexity of $P$.

The hypercube $[0,1]^n$ in $n$ dimensions has $2n$ facets and $2^n$ vertices. It is known that the extension complexity of the hypercube is equal to $2n$ and this was proved recently in \cite[Proposition 5.9]{fiorini2013combinatorial} using a combinatorial argument (in fact it was shown that any polytope in $\RR^n$ that is combinatorially equivalent to the hypercube has extension complexity $2n$). Note that the trivial lower bound obtained from the rank of the slack matrix in this case is $n+1$: in fact for any full-dimensional polytope $P \subset \RR^n$, the rank of the slack matrix of $P$ is equal to $n+1$, see e.g., \cite[Lemma 3.1]{gouveia2013polytopes}, and so the rank lower bound is in general not interesting in the context of slack matrices of polytopes. Also the lower bound of Goemans \cite{goemans2009smallest} here is $\log_2(2^n) = n$ which is not tight.

Below we use the lower bound on the nonnegative rank introduced in this paper to give another proof that the extension complexity of the hypercube in $n$ dimensions is $2n$.

\begin{proposition}
\label{prop:SlackHypercube}
Let $C_n = [0,1]^n$ be the hypercube in $n$ dimensions and let $S(C_n) \in \RR^{2n \times 2^n}$ be its slack matrix. Then
\[ \rank_+(S(C_n)) = \left(\frac{\nu_+^{[0]}(S(C_n))}{\|S(C_n)\|_F}\right)^2 = 2n. \]
\end{proposition}
\begin{proof}
The proof is given in Appendix \ref{sec:proofSlackHypercube}.
\end{proof}

\end{example}

\begin{example}[Matrix with strictly positive entries] 
\label{ex:stpos}
Consider the following matrix with strictly positive entries ($\epsilon > 0$):
\[ 
\setlength{\arraycolsep}{4pt}
A_{\epsilon} = \left[\begin{array}{cccc}
1+\epsilon & 1+\epsilon & \epsilon   & \epsilon\\
1+\epsilon & \epsilon   & 1+\epsilon & \epsilon\\
\epsilon   & 1+\epsilon & \epsilon   & 1+\epsilon\\
\epsilon   & \epsilon   & 1+\epsilon & 1+\epsilon
\end{array}\right]. \]
The matrix $A_{\epsilon}$ is obtained from the matrix of Example \ref{ex:cohenrothblum} by adding a constant $\epsilon > 0$ to each entry. The matrix $A_{\epsilon}$ has appeared before, e.g., in \cite[Equation 12]{gillis2012sparse} (under a slightly different form) and corresponds to the slack matrix for the pair of polytopes $[-1,1]^2$ and $[-1-2\epsilon,1+2\epsilon]^2$. It can be shown in this particular case, since $A_{\epsilon}$ is $4\times 4$ and $\rank A_{\epsilon} = 3$, that the nonnegative rank of $A_{\epsilon}$ is the smallest $r$ such that there is a polygon $P$ with $r$ vertices such that $[-1,1]^2 \subset P \subset [-1-2\epsilon,-1+2\epsilon]^2$, cf. \cite{gillis2012sparse} for more details.

Since $A_\epsilon$ is a small perturbation of the matrix $A$ of Example \ref{ex:cohenrothblum} one can use the approach\footnote{Since we are interested in obtaining a lower bound for the \emph{specific} perturbation $A_{\epsilon}$ of $A$, we can obtain a better bound than the one of Theorem \ref{thm:lbapproxnnrank} by normalizing by $\|A_{\epsilon}\|_F^2$ directly.} of Theorem \ref{thm:lbapproxnnrank} to get a lower bound on $\nu_+^{[0]}(A_{\epsilon})$. Let $W$ be the matrix defined in \eqref{eq:Wsquare}. Since $W$ is feasible for \eqref{eq:nu+0} we have $\nu_+^{[0]}(A_{\epsilon}) \geq \langle A_{\epsilon}, W \rangle$. It turns out in this example that the value of $\langle A_{\epsilon}, W \rangle$ does not depend on $\epsilon$ and is equal to $4\sqrt{2}$. This allows to obtain lower bounds on $\rank_+(A_{\epsilon})$ without solving any additional optimization problem. For example for $\epsilon = 0.1$ we get:
\[ \rank_+(A_{0.1}) \geq \frac{(4\sqrt{2})^2}{\|A_{0.1}\|_F^2} \approx 3.2 \]
which shows that $\rank_+(A_{0.1}) = 4$.
\end{example}

\begin{example}[Comparison with the Boolean rank lower bound]
\label{ex:rc}
We now show an example where the lower bound $(\nu_+^{[0]}(A)/\|A\|_F)^2$ is strictly greater than the rectangle covering lower bound (i.e., Boolean rank lower bound). Consider the $4 \times 4$ matrix
\[
\setlength{\arraycolsep}{3pt} 
A = \left[\begin{array}{rrrr}
0 & 1 & 1 & 1 \\
1 & 1 & 1 & 1 \\
1 & 1 & 0 & 0 \\
1 & 1 & 0 & 0 \end{array}\right]. \]
Note that the rectangle covering number of $A$ is 2 since $\support(A)$ can be covered with the two rectangles $\{1,2\}\times\{2,3,4\}$ and $\{2,3,4\}\times\{1,2\}$. If we compute the quantity $\nu_+^{[0]}(A)$ and the associated lower bound we get $\rank_+(A) \geq \lceil (\nu_+^{[0]}(A)/\|A\|_F)^2 \rceil = 3$ which is strictly greater than the rectangle covering number. In fact $\rank_+(A)$ is exactly equal to 3 since we have the factorization
{
\setlength{\arraycolsep}{3pt}
\[ 
 A = \left[\begin{array}{ccc}
1 & 1 & 0\\
1 & 1 & 1\\
0 & 1 & 1\\
0 & 1 & 1
\end{array}\right]
\left[\begin{array}{cccc}
0 & 0 & 1 & 1\\
0 & 1 & 0 & 0\\
1 & 0 & 0 & 0
\end{array}\right]. \]
}
\end{example}

\bigskip

The following simple proposition concerning diagonal matrices will be needed for the next example:
\begin{proposition}
\label{prop:diagonalMatrices}
If $A \in \RR^{n \times n}$ is a nonnegative diagonal matrix and $P$ is a diagonal matrix with strictly positive elements on the diagonal, then
\[ \nu_+(A;P,P) = \langle A, P \rangle = \sum_{i=1}^n A_{i,i} P_{i,i}. \]
In particular, for $P=I$ we get $\nu_+(A) = \trace(A)$.
\end{proposition}
\begin{proof}
We first show that $\nu_+(A;P,P) \leq \langle A, P \rangle$: Observe that if $W$ is such that $\begin{bmatrix} P & -W\\ -W & P\end{bmatrix}$ is copositive, then we must have $W_{i,i} \leq P_{i,i}$ for all $i\in\{1,\dots,n\}$ (indeed, take $e_i$ to be the $i$'th element of the canonical basis of $\RR^n$, then by copositivity we must have $2e_i^\tr P e_i - 2e_i^\tr W e_i \geq 0$ which gives $W_{i,i} \leq P_{i,i}$). Hence, since $A$ is diagonal we have $\langle A, W \rangle \leq \langle A, P \rangle$ for all feasible matrices $W$ and thus $\nu_+(A;P,P) \leq \langle A, P \rangle$.
Now if we take $W=P$, we easily see that $\begin{bmatrix} P & -W\\ -W & P\end{bmatrix}$ is copositive, and thus $\nu_+(A,P;P) \geq \langle A, P \rangle$.
Thus we conclude that $\nu_+(A;P,P) = \langle A, P \rangle$.
\end{proof}

\begin{example}[Matrix rescaling]
\label{ex:diagonal}
 The well-known lower bound on the rank
\[ \rank(A) \geq \left(\frac{\sigma_1(A)+\dots+\sigma_r(A)}{\sqrt{\sigma_1(A)^2+\dots+\sigma_r(A)^2}}\right)^2 \]
is sharp when all the singular values of $A$ are equal, but it is loose when the matrix is not well-conditioned. The same phenomenon is expected to happen also for the lower bound of Theorem \ref{thm:main} on the nonnegative rank, even if it is not clear how to define the notion of condition number in the context of nonnegative matrices. If one considers the following diagonal matrix where $\beta > 0$
\[ A = \begin{bmatrix}
\beta & 0 & \dots & 0\\
0 & 1 &  & \vdots\\
\vdots & &\ddots & 0\\
0 & \dots & 0 & 1
\end{bmatrix} \in \RR^{n \times n}_+, \]
then we have from Proposition \ref{prop:diagonalMatrices}, $\nu_+(A) = \trace(A) = \beta + n-1$, and thus for the lower bound we get:
\[ \rank_+(A) \geq \left(\frac{\nu_+(A)}{\|A\|_F}\right)^2 = \left(\frac{\beta+n-1}{\sqrt{\beta^2+n-1}}\right)^2. \]
If $\beta$ is large and grows with $n$, say for example $\beta = n$, the lower bound on the right-hand side is $\leq O(1)$ whereas $\rank_+(A) = n$. Note also that the standard rank of $A$ is equal to $n$.

To remedy this, one can use the more general lower bound of Theorem \ref{thm:main_gen} with weight matrices $P$ and $Q$ to obtain a better lower bound, and in fact a sharp one. Indeed, if we let
\[ P = Q = \begin{bmatrix}
\epsilon & 0 & \dots & 0\\
0 & 1 &  & \vdots\\
\vdots & &\ddots & 0\\
0 & \dots & 0 & 1
\end{bmatrix} \]
then we have $\nu_+(A;P,Q)^2 = (n-1+\beta\epsilon)^2$ and $\trace(A^\tr P A Q) = n-1+(\beta\epsilon)^2$. Hence for $\epsilon = 1/\beta$, the ratio $\left(\frac{\nu_+(A;P,Q)}{\sqrt{\trace(A^\tr P A Q)}}\right)^2$ is equal to $n$ which is equal to $\rank_+(A)$.
\end{example}

\begin{example}[Derangement matrix]
\label{ex:derangement} We now consider another example of a matrix that is not well-conditioned and where the lower bound of Theorem \ref{thm:main} is loose. Consider the derangement matrix
\[
\setlength{\arraycolsep}{3pt}  D_n = 
\left[\begin{array}{cccc}
0 & 1 & \dots & 1\\
1 & \ddots & \ddots & \vdots\\
\vdots & \ddots & \ddots & 1\\
1 &\dots & 1 & 0
\end{array}\right] \]
that has zeros on the diagonal and ones everywhere else. Then we have $\rank_+(D_n) = \rank(D_n) = n$, but we will show that 
\[ \left(\frac{\nu_+(D_n)}{\|D_n\|_F}\right)^2 \leq 4\]
for all $n$. Observe that the matrix $D_n$ is not well-conditioned since it has one singular value equal to $n-1$ while the other remaining singular values are all equal to 1.

To show that the lower bound of Theorem \ref{thm:main} is always bounded above by 4, note that the quantity $\nu_+(D_n)$ in this case is given by:
\begin{equation}
\label{eq:nu+Dn}
 \nu_+(D_n) = \max \; \left\{ \; \sum_{i \neq j} W_{i,j} \; : \; \begin{bmatrix} I & -W\\ -W^\tr & I \end{bmatrix} \text{ copositive} \; \right\} 
\end{equation}

Observe that by symmetry, one can restrict the matrix $W$ in the program above to have the form:
\footnote{Indeed, observe first that if $W$ is feasible for \eqref{eq:nu+Dn} then $W^\tr$ is also feasible and has the same objective value. Thus by averaging one can assume that $W$ is symmetric. Then note that for any feasible symmetric $W$ and any permutation $\sigma$, the new matrix $W'_{i,j} = W_{\sigma(i),\sigma(j)}$ is also feasible and has the same objective value as $W$. Hence again by averaging we can assume $W$ to be constant on the diagonal and constant on the off-diagonal.}
\[ 
\setlength{\arraycolsep}{3pt}
 W = \left[\begin{array}{cccc}
a & b & \dots & b\\
b & \ddots & \ddots & \vdots\\
\vdots & \ddots & \ddots & b \\
b & \dots & b & a
\end{array}\right] = bJ_n + (a-b) I_n \]
where $J_n$ is the $n\times n$ all ones matrix.
For $W$ of the form above we have for $u$ and $v$ arbitrary vectors in $\RR^n$, $u^\tr W v = (a-b)u^\tr v + b (1^\tr u)(1^\tr v)$ and hence the condition that $\begin{bmatrix} I_n & -W\\ -W^\tr & I_n \end{bmatrix}$ is copositive means that 
\[ \forall\;  u, v \in \RR^n_+,\;\;(a-b)u^\tr v + b (1^\tr u)(1^\tr v) \; \leq \; \frac{1}{2}(u^\tr u + v^\tr v) \]
Hence the quantity $\nu_+(D_n)$ can now be written as:
\[ \nu_+(D_n) =  \max_{a,b \in \RR} \; \left\{ \; (n^2-n) b \; : \; \forall u, v \in \RR^n_+,\; (a-b)u^\tr v + b (1^\tr u)(1^\tr v) \leq \frac{1}{2}(u^\tr u + v^\tr v) \; \right\} \]
Let us call $b_n$ the largest $b$ in the problem above, so that $\nu_+(D_n) = (n^2-n)b_n$. Assuming $n$ is even let 
\[ u_n = (\underbrace{1,\dots,1}_{n/2},0,\dots,0) \in \RR^n_+ \; \text{ and } \; v_n = (0,\dots,0,\underbrace{1,\dots,1}_{n/2}) \in \RR^n_+.\]
Then we have $u_n^\tr v_n = 0$ and the optimal $b_n$ must satisfy $b_n (1^\tr u_n)(1^\tr v_n) \leq \frac{1}{2}(u_n^\tr u_n + v_n^\tr v_n)$, which gives $b_n \frac{n^2}{4} \leq \frac{n}{2}$, i.e., $b_n \leq 2/n$. Hence $\nu_+(D_n) \leq (n^2-n) \cdot 2/n = 2(n-1)$ and 
\[ \left(\frac{\nu_+(D_n)}{\|D_n\|_F}\right)^2 \leq \frac{4(n-1)^2}{n^2 - n} = 4(1-1/n) \leq 4. \]
The case when $n$ is odd can be treated the same way and we also obtain the same upper bound of 4.
\end{example}

\paragraph{Summary of examples} In the examples above we have shown that our lower bounds (both the exact and the first relaxation) are uncomparable to most existing bounds on the nonnegative rank:
\begin{itemize}
\item Examples \ref{ex:cohenrothblum} and \ref{ex:derangement} show that the bound can be either larger or smaller than the standard rank.
\item Examples \ref{ex:rc} and \ref{ex:diagonal} show that the bound can be either larger or smaller than the rectangle covering number and the fooling set bound.
\end{itemize}
The following inequalities are however always satisfied for any nonnegative matrix $A$:
\[ \rank_+(A) \geq \left(\frac{\nu_+(A)}{\|A\|_F}\right)^2 \geq \left(\frac{\nu_+^{[0]}(A)}{\|A\|_F}\right)^2 \geq \left(\frac{\nu(A)}{\|A\|_F}\right)^2. \]

\section{Conclusion}

In this paper we have presented new lower bounds on the nonnegative
rank that can be computed using semidefinite programming. Unlike
many of the existing bounds, our lower bounds do not solely depend on
the sparsity pattern of the matrix and are applicable to matrices with
strictly positive entries.

An interesting question is to know whether the techniques
presented here can be strengthened to obtain sharper lower bounds. In
particular the results given here rely on the ratio of the $\ell_1$
and the $\ell_2$ norms, however it is known that the ratio of the
$\ell_1$ and the $\ell_p$ norms for $1 < p < 2$ can yield better lower
bounds. 

Another question left open in this paper is the issue of scaling. As we saw in Section \ref{sec:lb} and in the examples, one can choose scaling matrices $P$ and $Q$ to improve the bound. It is not clear however how to compute the optimal scaling $P$ and $Q$ that yields the best lower bound. In recent work \cite{fawzi2014self} we extend some of the ideas presented in this paper to obtain a new lower bound on the nonnegative rank which is invariant under diagonal scaling, and which also satisfies other interesting properties (e.g., subadditivity, etc.). In fact the technique we propose in \cite{fawzi2014self} applies to a large class of atomic cone ranks and can be used for example to obtain lower bounds on the cp-rank of completely positive matrices \cite{berman2003completely}.

Finally it is natural to ask whether the ideas presented here can be applied to obtain lower bounds on the \emph{positive semidefinite (psd) rank}, a quantity which was introduced recently in \cite{gouveia2011lifts} in the context of positive semidefinite lifts of polytopes. One main difficulty however is that the psd rank is not an ``atomic'' rank, unlike the nonnegative rank where the atoms correspond to nonnegative rank-one matrices. In fact it is this atomic property of the nonnegative rank which was crucial here to obtain the lower bounds in this paper.

\appendix

\section{Proof of Proposition \ref{prop:SlackHypercube}: Slack matrix of hypercube}
\label{sec:proofSlackHypercube}

In this appendix we prove Proposition \ref{prop:SlackHypercube}
concerning the nonnegative rank of the slack matrix of the
hypercube. We restate the proposition here for convenience:

\begin{propstar}
Let $C_n = [0,1]^n$ be the hypercube in $n$ dimensions and let $S(C_n) \in \RR^{2n \times 2^n}$ be its slack matrix. Then
\[ \rank_+(S(C_n)) = \left(\frac{\nu_+^{[0]}(S(C_n))}{\|S(C_n)\|_F}\right)^2 = 2n. \]
\end{propstar}
\begin{proof}
The facets of the hypercube $C_n = [0,1]^n$ are given by the linear inequalities $\{x_k \geq 0\}, \; k=1,\dots,n$ and $\{x_k \leq 1\}, \; k=1,\dots,n$, and the vertices of $C_n$ are given by the set $\{0,1\}^n$ of binary words of length $n$. It is easy to see that the slack matrix of the hypercube is a 0/1 matrix: in fact, for a given facet $F$ and vertex $V$, the $(F,V)$'th entry of $S(C_n)$ is given by:
\[ S(C_n)_{F,V} = \begin{cases} 1 & \text{ if $V \notin F$}\\ 0 & \text{ if $V \in F$} \end{cases}. \]
Since the slack matrix of the hypercube $S(C_n)$ has $2n$ rows we clearly have 
\[ \left(\frac{\nu_+^{[0]}(S(C_n))}{\|S(C_n)\|_F}\right)^2 \leq \rank_+(S(C_n)) \leq 2n. \]
To show that we indeed have equality, we exhibit a particular feasible point $W$ of the semidefinite program \eqref{eq:nu+0} and we show that this $W$ satisfies $(\langle S(C_n), W \rangle / \|S(C_n)\|_F)^2 = 2n$. 

Let 
\begin{equation}
\label{eq:Woptcube}
 W = \frac{1}{\sqrt{2^{n-1}}} (2 S(C_n) - J).
\end{equation}
Observe that $W$ is the matrix obtained from $S(C_n)$ by changing the ones into $\frac{1}{\sqrt{2^{n-1}}}$ and zeros into $-\frac{1}{\sqrt{2^{n-1}}}$. For this $W$  we verify using a simple calculation that 
\[ \left(\frac{\langle S(C_n), W \rangle}{\|S(C_n)\|_F}\right)^2 = 2n. \]
It thus remains to prove that $W$ is indeed a feasible point of the semidefinite program \eqref{eq:nu+0} and that the matrix
\[ \begin{bmatrix} I & -W\\ -W^\tr & I\end{bmatrix} \in \RR^{(2n + 2^n)\times (2n+2^n)} \]
 can be written as the sum of a nonnegative matrix and a positive semidefinite one. This is the main part of the proof and for this we introduce some notations. Let $\cF$ be the set of facets of the hypercube, and $\cV=\{0,1\}^n$ be the set of vertices.  If $F \in \cF$ is a facet of the hypercube, we denote by $\bar{F}$ the opposite facet to $F$ (namely, if $F$ is given by $x_k \geq 0$, then $\bar{F}$ is the facet $x_k \leq 1$ and vice-versa). Similarly for a vertex $V \in \cV$ we denote by $\bar{V}$ the opposite vertex obtained by complementing the binary word $V$.
 Denote by $N_{\cF}:\RR^{\cF}\rightarrow \RR^{\cF}$ and $N_{\cV}:\RR^{\cV}\rightarrow \RR^{\cV}$ the ``negation'' maps, so that we have:
 \begin{equation} \label{eq:negF} \forall g \in \RR^{\cF},\; \forall F \in \cF,\; (N_{\cF} g)(F) = g(\bar{F}) \end{equation}
 \begin{equation} \label{eq:negV} \forall h \in \RR^{\cV},\; \forall V \in \cV,\; (N_{\cV} h)(V) = h(\bar{V}) \end{equation}
Note that in a suitable ordering of the facets and the vertices, the matrix representation of $N_{\cF}$ and $N_{\cV}$ take the following antidiagonal form ($N_{\cF}$ is of size $2n\times 2n$ and $N_{\cV}$ is of size $2^n \times 2^n$):
  \[ \begin{bmatrix}
0 &         &           & 1\\
  &         & \iddots   &  \\
  & \iddots &           &  \\
1 &         &           &  0
\end{bmatrix} \]
Consider now the following decomposition of the matrix $\begin{bmatrix} I & -W\\ -W^\tr & I\end{bmatrix}$:
 \begin{equation}
 \label{eq:decompositionSlackHypercube}
\begin{bmatrix} I & -W\\ -W^\tr & I \end{bmatrix} = \begin{bmatrix} N_{\cF} & 0\\ 0 & N_{\cV} \end{bmatrix} + \begin{bmatrix} I-N_{\cF} & -W\\ -W^\tr & I-N_{\cV} \end{bmatrix}
  \end{equation}
 Clearly the first matrix in the decomposition is nonnegative. The next lemma states that the second matrix is actually positive semidefinite: 
 
\begin{lemma}
Let $\mathcal F$ and $\mathcal V$ be respectively the set of facets and vertices of the hypercube $C_n=[0,1]^n$ and let $\hat{W} \in \RR^{\cF \times \cV}$ be the matrix:
\[ \hat{W}_{F,V} = \begin{cases} 1 & \text{ if $V \notin F$}\\ -1 & \text{ if $V \in F$} \end{cases}\;\; \forall F \in \cF,\; V \in \cV \]
Then the matrix
\begin{equation}
 \label{eq:psdmatrixhypercube}
 \begin{bmatrix} I-N_{\cF} & -\gamma \hat{W}\\ -\gamma\hat{W}^\tr & I-N_{\cV} \end{bmatrix}
\end{equation}
is positive semidefinite for $\gamma = 1/\sqrt{2^{n-1}}$ (where $N_{\cF}$ and $N_{\cV}$ are defined in \eqref{eq:negF} and \eqref{eq:negV}).
\end{lemma}

\begin{proof}
We use the Schur complement to show that the matrix \eqref{eq:psdmatrixhypercube} is positive semidefinite. In fact we show that
\begin{enumerate}
\item $I-N_{\cV} \succeq 0$
\item $\range(\hat{W}^\tr) \subseteq \range(I-N_{\cV})$, and 
\item $I-N_{\cF} - \gamma^2 \hat{W} (I-N_{\cV})^{-1} \hat{W}^\tr \succeq 0$.
\end{enumerate}
where $(I-N_{\cV})^{-1}$ denotes the pseudo-inverse of $I-N_{\cV}$.

Observe that for any $k \in \NN$, the $2k \times 2k$ matrix given by:
\[ I_{2k} -N_{2k} = \begin{bmatrix}
1  &    &     & -1\\
   &  \ddots   & \iddots   &  \\
   & \iddots & \ddots  &  \\
-1 &        &        &  1
\end{bmatrix} \]
is positive semidefinite: in fact one can see that $\frac{1}{2}(I_{2k}-N_{2k})$ is the orthogonal projection onto the subspace spanned by its columns (i.e., the subspace of dimension $k$ spanned by $\{e_i - e_{2k-i} \; : \; i=1,\dots,n\}$ where $e_i$ is the $i$'th unit vector). Hence this shows that $I-N_{\cV}$ is positive semidefinite, and it also shows that $(I-N_{\cV})^{-1} = \frac{1}{4} (I-N_{\cV})$.

Now we show that $\range(\hat{W}^\tr) \subseteq \range(I-N_{\cV})$. For any $F \in \cF$, the $F$'th column of $\hat{W}^\tr$ satisfies $(\hat{W}^\tr)_{V,F} = -(\hat{W}^\tr)_{\bar{V},F}$ for any $V \in \cV$, and thus $\range(\hat{W}^\tr) \subseteq \linspan(e_{V} - e_{\bar{V}}, \; : \; V \in \cV) = \range(I-N_{\cV})$.

It thus remains to show that
\[ I-N_{\cF} - \gamma^2 \hat{W} (I-N_{\cV})^{-1} \hat{W}^\tr \succeq 0 \]
First note that since $\frac{1}{2} (I-N_{\cV})$ is an orthogonal projection and that $\range(\hat{W}^\tr) \subseteq \range(I-N_{\cV})$, we have $(I-N_{\cV})^{-1} \hat{W}^\tr = \frac{1}{2} \hat{W}^\tr$. Thus we now have to show that
\[ I-N_{\cF} - \frac{\gamma^2}{2} \hat{W} \hat{W}^\tr \succeq 0. \]
The main observation here is that the matrix $\hat{W}\hat{W}^\tr$ is actually equal to $2^{n} (I-N_{\cF})$. For any $F,G \in \cF$, we have:
\[ (\hat{W}\hat{W}^\tr)_{F,G} = \sum_{a \in \cV}\hat{W}_{F,a} \hat{W}_{G,a} = \begin{cases} 2^n & \text{ if } F=G\\ -2^n & \text{ if } F=\bar{G}\\ 0 & \text{ else} \end{cases} \]
First it is clear that if $F=G$, then $(\hat{W}\hat{W}^\tr)_{F,G} = 2^n$. Also if $F=\bar{G}$ then $(\hat{W}\hat{W}^\tr)_{F,G} = -2^n$ since if $F=\bar{G}$ then $a \in F \Leftrightarrow a \notin G$ hence $\hat{W}_{F,a} \hat{W}_{G,a} = -1$ for all $a \in \cV$. In the case that $F \neq G$ and $F \neq \bar{G}$, it is easy to verify by simple counting that $\sum_{a \in \cV}\hat{W}_{F,a} \hat{W}_{G,a} = 0$.

Hence we have $I-N_{\cF} - \gamma^2 \hat{W} (I-N_{\cV})^{-1} \hat{W}^\tr = I-N_{\cF} - \frac{\gamma^2}{2} \hat{W}\hat{W}^\tr = (1 - \frac{\gamma^2}{2}2^n) (I-N_{\cF})$ which is positive semidefinite for $\gamma = 1/\sqrt{2^{n-1}}$.

\end{proof}
Using this lemma, Equation~\eqref{eq:decompositionSlackHypercube} shows that the matrix $W$ is feasible for the semidefinite program \eqref{eq:nu+0}, and thus that $\nu_+^{[0]}(S(C_n)) \geq \langle S(C_n), W \rangle = \sqrt{2^{n-1}} \cdot 2n$. Hence since $\|S(C_n)\|_F = \sqrt{2^{n-1} \cdot 2n}$ we get that 
\[ \rank_+(S(C_n)) \geq \left(\frac{\nu_+^{[0]}(S(C_n))}{\|S(C_n)\|_F}\right)^2 \geq  \left(\frac{\sqrt{2^{n-1}} \cdot 2n}{\sqrt{2^{n-1} \cdot 2n}}\right)^2 = 2n \]
which completes the proof.
\end{proof}

\bibliographystyle{alpha}
\bibliography{nonnegative_rank}

\end{document}